\documentclass{amsart}
\usepackage{amsmath}
\usepackage{amssymb}
\usepackage{graphicx,float} 

\newtheorem{theorem}{Theorem}[section]

\newtheorem{lemma}{Lemma}[section]

\def\ad#1{\begin{aligned}#1\end{aligned}}  \def\b#1{{\mathbf{#1}}}
\def\a#1{\begin{align*}#1\end{align*}} \def\an#1{\begin{align}#1\end{align}} 
 \def\t#1{{\operatorname{#1}}}
  \numberwithin{equation}{section}

\begin{document} \baselineskip=16pt\parskip=4pt

\title[finite element]
 {Convergent finite elements on arbitrary meshes, the WG method}
  
\author { Ran Zhang }
\address{School of Mathematics, Jilin University, Changchun, 130012, China}
\email{zhangran@jlu.edu.cn}

\author { Shangyou Zhang }
\address{Department of Mathematical             Sciences, University
     of Delaware, Newark, DE 19716, USA. }
\email{szhang@udel.edu }
  
%\date{}



\begin{abstract} On meshes with the maximum angle condition violated, the
  standard conforming, nonconforming, and discontinuous Galerkin finite elements
  do not converge to the true solution when the mesh size goes to zero.
  It is shown that one type of weak Galerkin finite element method converges
     on triangular and tetrahedral meshes violating the maximum angle condition,
     i.e., on arbitrary meshes.
  Numerical tests confirm the theory.
 
\end{abstract}

\maketitle

%%%%%%%%%%%%%%%%%%%%%%%%%%%%%%%%%%%%%%%%%%%%%%%%%%%%%%%%%%%%%%%%%%%%%%%%%%%
\section{Introduction}
In one-dimensional finite element computation, the discrete solutions converge to the
  true solution as long as the size of the largest interval goes to zero, no matter how
  the mesh points are placed.  
But it is no longer true in  higher dimensions.
A famous example was found by Babu\v{s}ka and Aziz in \cite{Babuska} that the
  finite element solution does not converge to the true solution on some
   patterned triangular meshes (cf. Figure \ref{f-c-c}(A))
   which violate the maximum angle condition.
   
The problem can be illustrated by two examples in Calculus.
When one approximates a circle by polygons,  no matter how the interpolation points are placed,
  the arc-length of the polygon converges to the circumference of the circle, cf. Figure \ref{f-c-c}(E),
    as long as the longest edge goes to zero.
In Figure \ref{f-c-c}(B), one does the problem in two dimensions,  
  interpolating a cylinder by linear triangles with a horizontal size $h$ and a
  vertical size $h^2$.  When $h\to 0$, though the distance between the cylinder and the triangles 
    goes to zero, 
  the total area of triangles converges to a bigger number different from the surface area of the cylinder.
One can see that the small triangles keep an angle to the cylinder.
In other words, they do not approach the tangent planes to the surface.

\begin{figure}[H]
 \begin{center}\setlength\unitlength{1.0pt}
\begin{picture}(360,225)(-10,0)
  \put(150,0){\includegraphics[width=180pt]{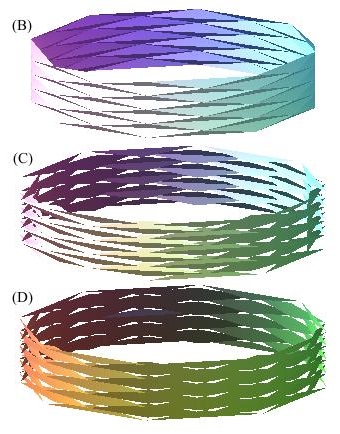}} 
  \put(25,28){\includegraphics[width=140pt]{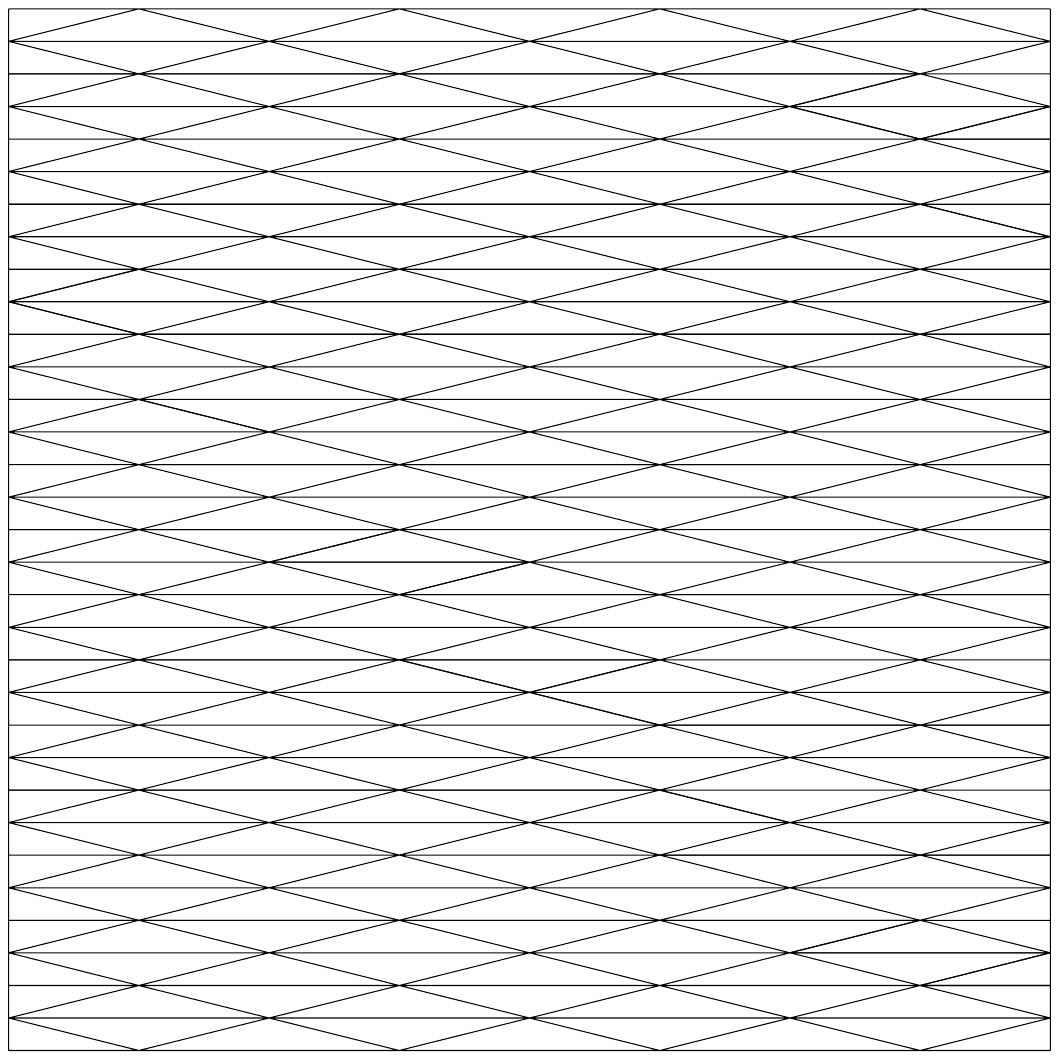}}
  \put(35,10){\includegraphics[width=90pt]{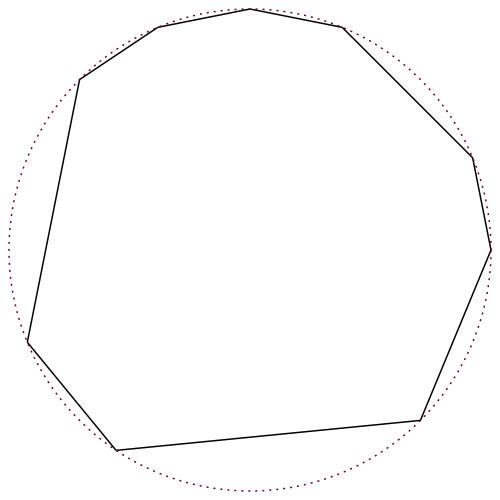}} \put(15,200){\small(A)}\put(23,80){\small(E)}
 \end{picture}\end{center}
\caption{(A) A triangular mesh violates the maximum angle condition, when $h\to 0$.
  \quad (B) The area of triangles (interpolating at three vertices) converges to different a number.  
  \quad (C) The area of triangles (interpolating at three mid-edge points) converges to different a number. 
  \quad (D) The area of triangles ($L^2$-projection) converges to that of the cylinder. 
  \quad (E) The arc-length of line segments converges to that of the circle when mesh size goes to zero. }
\label{f-c-c}
\end{figure}

In addition to continuous finite element methods, there is a type of elements \cite{Crouzeix},
   called nonconforming
  finite elements, whose functions are continuous up to the $P_{k-1}$ order, i.e.,
    the jump of two $P_k$ polynomials on two sides of an edge 
     is $L^2$-orthogonal to the space of $P_{k-1}$ polynomials on the edge.
Translating this method to the above cylinder approximation problem,
   one does the $P_1$ interpolation at three mid-edge points, instead of the
     three vertex points, cf. Figure \ref{f-c-c}(C).
Though we have a slightly better approximation,
   the problem is not solved as the area of triangles still
     does not converge to that of the cylinder.
Similarly, the non-conforming finite element solutions do not converge to the
true solution on such meshes.

But the problem of approximating the cylinder by triangles can be easily solved if we further
  move the three interpolation points toward the center as the resulting secant triangles become almost 
   tangent triangles, cf. Figure \ref{f-c-c}(D).
Or better, we can uses the closest triangles on the mesh, i.e., the $P_1$ $L^2$-projection of
   the cylinder function. 
Then the total area of triangles would converge to the area of the cylinder.
Applying the idea to the finite element methods, can we allow arbitrary meshes in
  the discontinuous Galerkin (DG) 
  (discontinuous piecewise polynomials \cite{Arnold, Arnold2,Babuska73}) method?
Instead of the above mentioned local $L^2$-projection, 
  the finite element solution is the global $H^1$-projection when solving the Poisson equation.
In DG methods, inter-element integrals and penalties are introduced 
   to control the inconsistency caused by the inter-element discontinuity.
Hence the finite element solutions are almost
  continuous and fail to converge to the correct solution if the maximum angle condition does not hold.
 
Can we make the discontinuous solutions less ``continuous", for example,
   close to the element-wise $L^2$-projection, so that the method works on arbitrary meshes? 
The answer is yes.  
We know that the $L^2$-projection converges on arbitrary meshes.
We need a discontinuous finite element method for which the $H^1$ projection is super-close
  to the $L^2$-projection.

A type of superconvergent, weak Galerkin (WG) finite element methods 
  \cite{Al-Taweel, Mu,
   yz1,yz2,yz3, yz4,yz5,yz6} is designed recently, 
   where there is no inter-element integral, neither any penalty term,
   in the discrete equation.
The finite element solution is almost the $L^2$-projection. 
In this work,  we show that this $P_k$ WG finite element method produces two order
  superconvergent solutions on arbitrary triangular and tetrahedral meshes.

One may argue that no such bad meshes are and will ever be used.
In fact, it is difficult for us to code one example of such degenerate tetrahedral meshes.
But, for example, to avoid one bad-shape tetrahedron, a mesh generator or an adaptive algorithm would subdivide
  a large region excessively to very tiny tetrahedra, even to reach the computer limit.
If we do not need this maximum angle condition,  then we can allow some isolated bad-shape 
  elements to save computation.
That is, the mesh quality index would be revised, dropping few bad quality elements.

This work is significant in mathematics, as it is the first time ever we can guarantee the
  convergence of finite element solutions on meshes without any condition.
The analysis is   different somewhat from the typical error estimate in the finite element method.
For example, the inter-element error is converted to other errors inside elements.
The local interpolation and projection are introduced directly to the error estimation
  of the $H^1$-projection, while they are typically used to show the subspace has
   the required global approximation property only.
Also a gradient-preserving Zhang-Zhang transformation is defined, 
  which is like a scalar version of inverse (divergence-preserving) Piola transformation.

%%%%%%%%%%%%%%%%%%%%%%%%%%%%%%%%%%%%%%%%%%%%%%%%%%%%%%%%%%%%%%%%%%%%%%%%%%%
\section{The weak Galerkin finite element method}

Let $\mathcal{T}_h=\{T \}$ be an arbitrary triangular or tetrahedral mesh of size $h$ 
    on a polygonal or a polyhedron domain $\Omega$.  
Let $\mathcal E_h=\{e\}$ be the set of all edges or face-triangles in $\mathcal T_h$.
Let the WG finite element space on the mesh $\mathcal T_h$ be
\an{\label{V-h} \ad{
    V_h = \{ v_h=\{v_0,v_b\} : v_0&\in L^2(\mathcal T_h), \ v_0|_T \in P_k(T), \\
                               v_b&\in L^2(\mathcal E_h), \ v_b|_e\in P_{k+1}(e), \
                                v_b|_{\partial\Omega}=0 \}, } }
where $P_1(T)$ denotes the space of polynomials of degree $1$ or less on $T$.
We note that $v_h$ is a generalized function and that $v_b$ is single-valued on $e$ between two
   elements.

 We define the weak gradient for the generalized functions in $V_h$ of \eqref{V-h} by
   $\nabla_w v_h\in [P_{k+1}(T)]^d$, satisfying
 \an{\label{w-g} (\nabla_w v_h, \b q)_{T} =
                  -(v_0, \nabla\cdot \b q)_T + \sum_{e\in\partial T}
                     \langle  v_b, \b q\cdot \b n \rangle_e
                     \quad \ \forall \b q\in [P_{k+1}(T)]^d, }
where $d=2$ or $3$, is the space dimension, $\b n$ is the outer unit vector on $e\subset\partial T$. 
By an integration by parts, we rewrite the definition \eqref{w-g} as 
 \an{\label{w-g2} (\nabla_w v_h-\nabla v_0, \b q)_{T} = 
          \langle   v_b -v_0 , \b q\cdot \b n \rangle_{\partial T}
                     \quad \ \forall \b q\in  [P_{k+1}(T)]^d. }

 We solve a model problem of Poisson's equation,
 \an{\label{pde} -\Delta u = f \quad \t{ in } \ \Omega; \quad u=0\quad \t{ on } \ \partial\Omega. }
The finite element approximation problem reads:  Find $u_h\in V_h$ such that
\an{\label{finite} (\nabla_w u_h, \nabla_w v_h) = ( f, v_h) \quad \forall v_h\in V_h, }
where $V_h$ is defined in \eqref{V-h} and $\nabla_w$ is defined in \eqref{w-g}.

\begin{lemma} For all $v_h=\{v_0,v_b\} \in V_h$ of \eqref{V-h}, it holds that
\an{ \label{0-b} \|\nabla v_0\|_0 \le C \|\nabla_w v_h\|_0,  }
where $C>0$ is independent of the shape regularity and the size of $\mathcal T_h$,
   $\|\cdot\|_0$ is the $L^2$ norm, and $\nabla_w$ is defined in \eqref{w-g}.
\end{lemma}  

\begin{proof}  On the unit reference triangle or tetrahedron $\hat T$,  we decompose 
   $v_h=\{v_0,v_b\}=\{v_0,v_0\} + \{0,v_b-v_0\}=v_1+v_2 \in V_1(\hat T) + V_2(\hat T)$.
   As $v_1$ is a one piece polynomial, $\hat \nabla_w v_1=\hat \nabla v_0$.
   Then, we have
\an{\label{t-0} \ad{
  \|\hat \nabla_w v_h\|^2_{\hat T}
     & =\|\hat \nabla v_0\|^2_{\hat T} +2 (\hat \nabla v_0,\hat \nabla_w v_2)_{\hat T}
         +\|\hat \nabla_w v_2\|_{\hat T}^2\\
     & \ge \|\hat \nabla v_0\|^2_{\hat T} -2\gamma \|\hat \nabla v_0\|_{\hat T}\|\hat \nabla_w v_2\|_{\hat T}
         +\|\hat \nabla_w v_2\|_{\hat T}^2\\
     & \ge (1-\gamma)\|\hat \nabla v_0\|^2_{\hat T}, }
} where $\gamma=\cos \theta < 1$ and $\theta$ is the angle between the two finite dimensional vector spaces
  under the inner product.  We map next $\hat T$ to a general triangle/tetrahedron by an affine (linear plus
  a shift) mapping $F$.  Let $J(F)$ be the Jacobi matrix of $F$.
  Then $M=J(T)^{-T}|J|^{1/2}$ is a constant $d\times d$  matrix.
  The mapped functions for $v_0$ and $v_2$ are $\hat v_0(\hat {\b x})=v_0(F(\hat {\b x}))$
     and $\hat v_2(\hat {\b x})=v_2(F(\hat {\b x}))$, respectively.
We define two polynomials on $\hat T$ by 
\an{\label{Zhang}  u_0( \hat{\b x})=\hat v_0(M\hat{\b x})\in V_1(\hat T)\quad \ \t{and} \ \quad
      u_2( \hat{\b x}) =\hat v_2(M\hat{\b x}) \in V_2(\hat T).}
It follows that
\a{ \hat \nabla  u_0 (\hat{\b x}) = M \hat \nabla v_0(\hat{\b x})\quad \ \t{and} \ \quad
     \hat \nabla  u_2 (\hat{\b x}) = M \hat \nabla v_2( \hat{\b x}).}      
We note that the gradient-preserving transformation \eqref{Zhang} is called a Zhang-Zhang
  transformation, which is like a scalar version of inverse (divergence-preserving)
  Piola transformation.
 By choosing two different functions and \eqref{t-0}, we have
\an{\label{t-g} \ad{
  \|\nabla_w v_h\|^2_{T}
     & =\| \nabla v_0\|^2_{T} +2 (M \hat \nabla \hat v_0, M \hat \nabla_w \hat v_2)_{\hat T}
         +\|\nabla_w v_2\|_{T}^2\\
     & =\| \nabla  v_0\|^2_{T} +2 ( \hat \nabla  u_0,  \hat \nabla_w   u_2)_{\hat T}
         +\|\nabla_w v_2\|_{T}^2\\
     & \ge \| \nabla  v_0\|^2_{T} -2\gamma \|\hat \nabla  u_0\|_{\hat T}\| 
             \hat \nabla_w   u_2\|_{\hat T}
         +\|\nabla_w v_2\|_{T}^2\\
     & = \| \nabla v_0\|^2_{T} -2\gamma \|M \hat \nabla \hat v_0\|_{\hat T}\| 
             M \hat \nabla_w \hat v_2\|_{\hat T}
         +\|\nabla_w v_2\|_{T}^2\\
           & \ge (1-\gamma)\| \nabla v_0\|^2_{ T}. }
} 
 Summing up \eqref{t-g} over all $T$, the lemma is proved.

\end{proof}
   
\begin{lemma}The linear system of finite element equations \eqref{finite}
   has a unique solution.
\end{lemma}  

\begin{proof} For a square system of finite linear equations,  we only need to prove the uniqueness.
  Let $ f=  0$ in \eqref{finite}.
Letting $ v_h=  u_h$ in \eqref{finite}, 
      we  get $\nabla_w u_h=0$.  By \eqref{0-b},  $\nabla u_0=0$.
Thus $u_0$ is a piece-wise constant function on the mesh.
Letting $\b q\in [P_{k+1}(T)]^d$ such that $\b q\cdot \b n=u_b-u_0$ on one edge/triangle and 
             $\b q\cdot \b n= 0$ on the rest face edges/triangles of $T$,
       by the definition of $\nabla_w$ in \eqref{w-g2}, we get 
\a{  0 &= (\nabla_w u_h, \b q)_T  =  (\nabla u_0,  \b q)
     +  \langle  u_b-u_0, \b q \cdot \b n \rangle_{\partial T} \\
  &= 0 +  \| u_b -u_0\|_{0,e}^2.  }  
     On the two sides of $e$, the two $u_0|_{e}= u_b$.
     Therefore $u_h$ is continuous on the
       whole domain.
 We conclude that $u_h$ is one global constant and it is zero by the boundary condition \eqref{V-h}.
The lemma is proved. 
\end{proof}
         
         \section{Convergence}
         
\begin{theorem}
Let $u\in (H^1_0(\Omega) \cap H^{k+3}(\Omega))$
    be the solution of the Poisson equation \eqref{pde}. 
Let $u_h\in V_h $  be the solution of the finite element problem   \eqref{finite}.
It holds that   
\an{ \label{h1} 
   h \| \Pi_{k+1}\nabla u - \nabla_w {u}_{h} \|_{0} + \|Q_h u - {u}_{h}\|_{0} & \le  C h^{k+3} |u|_{k+3}, } 
where $\Pi_{k+1}$ is the $L^2$ projection to the piecewise $P_{k+1}(T)$ space on mesh $\mathcal T_h$,
  and $Q_h w=\{w_0, w_b\}=\{\Pi_k w, \Pi_{k+1}^e w\}$ is the local $L^2$ projection on each element
    and each face edge/triangle.
\end{theorem}
\begin{proof} 
Testing the Poisson equation \eqref{pde} by a function $v_h=\{v_0,v_b\}\in V_h$, it follows by doing
 an integration by parts  and \eqref{w-g2} with $\b q=\Pi_{k+1}\nabla u$ that
\an{\label{p1} \ad{
    (f,v_h)  
    &= ( \nabla u, \nabla v_0) - \sum_{T\in \mathcal T_h} \langle
            \nabla u\cdot \b n, v_0 \rangle_{\partial T} \\ 
    &= (\Pi_{k+1} \nabla u, \nabla v_0) - \sum_{T\in \mathcal T_h} \langle
            \nabla u\cdot \b n, v_0 - v_b\rangle_{\partial T} \\
    &= ( \Pi_{k+1}  \nabla u, \nabla_w v_0) +\sum_{T\in \mathcal T_h}   \langle
            ( \Pi_{k+1}  \nabla u-\nabla u)\cdot \b n, v_0- v_b \rangle_{\partial T}. } }
By the definition of weak gradient \eqref{w-g} with $\b q=\Pi_{k+1}\nabla u-\nabla_w Q_h u$,
\an{\label{p-c} \ad{ \|\b q\|_0^2 &= (\Pi_{k+1}\nabla u-\nabla_w Q_h u, \b q) \\
       &=(\nabla u,\b q)+(\Pi_k u, \nabla\cdot \b q) - \sum_{T\in\mathcal T_h}
          \langle \Pi_{k+1}^e u, \b q\cdot\b n\rangle_{\partial T} \\
       &=(\Pi_k u-u, \nabla\cdot \b q) + \sum_{T\in\mathcal T_h}
          \langle u- \Pi_{k+1}^e u, \b q\cdot\b n\rangle_{\partial T} \\
       &= 0 + \sum_{T\in\mathcal T_h} 0 = 0.  } } 
Subtracting \eqref{p1} from \eqref{finite},  we get by \eqref{p-c}, \eqref{0-b} and 
   letting $\b q=\Pi_{k+1} \nabla  u - I_{\t{BDM}} \nabla  u\in [P_{k+1}]^d$ in  \eqref{w-g2} that
\an{\label{p2} \ad{ & \quad \  (\Pi_{k+1} \nabla u - \nabla_w {u}_{h}, \nabla_w v_h) 
              = (\nabla_w (Q_h u- u_h), \nabla_w v_h) \\
         &= -\sum_{T\in \mathcal T_h}   \langle
            ( \Pi_{k+1} \nabla u-I_{\t{BDM}} \nabla u)\cdot \b n, v_b-v_0 \rangle_{\partial T} 
     \\    &=\sum_{T\in \mathcal T_h}  
            (\Pi_{k+1} \nabla u-I_{\t{BDM}} \nabla u), \nabla v_0 -  \nabla_w v_h)_{T} 
      \\   &\le \bigg(\sum_{T\in \mathcal T_h}  
            2 \|\Pi_{k+1} \nabla u-\nabla u\|_0^2 
             + 2\|\nabla u-I_{\t{BDM}} \nabla u\|_0^2\bigg)^{1/2} \\
          & \quad \ \cdot \bigg(\sum_{T\in \mathcal T_h}  
            2 \| \nabla v_0\|_0^2 + 2 \|\nabla_w v_h \|_0^2\bigg)^{1/2}
        \\  &\le C h^{k+2} |u|_{k+3} \|\nabla_w v_h\|_0, } }
where $I_{\t{BDM}} : [H^1(T)]^d \to \t{BDM_{k+1}} =\{\b q\in H(\t{div}) : \b q|_T\in [P_{k+1}]^d\}$
   is the BDM interpolation defined by
\a{ (I_{\t{BDM}} \b w - \b w, \nabla q)_T &=0 \quad \forall \b q\in P_{k}(T)\setminus P_{0}(T), \\
    \langle (I_{\t{BDM}} \b w-\b w)\cdot\b n, q\rangle_{e} &=0
      \quad \forall q\in P_{k+1}(e),\\
      ( I_{\t{BDM}} \b w - \b w, \b q)_T&=0 \quad \forall \b q\in\{ \b q\in [P_{k+1}(T)]^d :
         \nabla\cdot \b q=0, \ \b q\cdot\b n|_{\partial T}=0\} . }
As $T$ is a star-shaped region, by local Taylor polynomials \cite{Brenner}, we have 
\a{ \|\Pi_{k+1} \nabla u-\nabla u\|_0\le C h^{k+2} |u|_{k+3}\quad \t{and } \quad
       \|\nabla u-I_{\t{BDM}} \nabla u\|_0 \le C h^{k+2} |u|_{k+3}, }
       which are used in \eqref{p2}.   
By \eqref{p2}, as $Q_h u\in V_h$, we can cancel the weak gradient from both sides to get
\an{\label{h1e}   \| \Pi_{k+1}\nabla u - \nabla_w {u}_{h} \|_{0}  & \le  C h^{k+2} |u|_{k+3}. } 

For the $L^2$ error estimate, we assume  the standard full-regular for the problem of finding
  $w\in H^1_0(\Omega)$ satisfying
\an{\label{w} (\nabla w, \nabla v) = (Q_h u-u_h, v) \quad \forall v \in H^1_0(\Omega), }
that 
\an{\label{reg}  |w |_2 \le C \|Q_h u-u_h\|_0.  }
Let $w_h\in V_h$ be the finite element solution for \eqref{w}.
Denoting $e_h=Q_h u- u_h=\{e_0,e_b\}$, by \eqref{p1} and  \eqref{p2} (up to $k=-1$ order for the
   third term below), we have
\a{ \|e_h\|_0^2&=(\nabla  w, \nabla e_0) -\sum_{T\in\mathcal T_h} 
     \langle \nabla w\cdot \b n, e_0-e_b\rangle_{\partial T}\\
     &=(\nabla_w Q_h w, \nabla_w e_h) +\sum_{T\in\mathcal T_h} 
     \langle (\Pi_{k+1} \nabla w- \nabla w)\cdot \b n, e_0-e_b\rangle_{\partial T}\\
    &= \sum_{T\in\mathcal T_h} ( \langle (\Pi_{k+1} \nabla u- \nabla u)\cdot \b n, 
       \Pi_k w- \Pi_{k+1}^e w \rangle_{\partial T} \\
    &\qquad\quad  + \langle (\Pi_{k+1} \nabla w- \nabla w)\cdot \b n, e_0-e_b\rangle_{\partial T})  \\
    & =\sum_{T\in \mathcal T_h}  ( 
            \nabla\cdot (\Pi_{k+1} \nabla u-I_{\t{BDM}} \nabla u), w -  \Pi_k w )_{T} \\
    &\qquad\quad  +  (\Pi_{k+1} \nabla u-I_{\t{BDM}} \nabla u), \nabla( w -  \Pi_k w) )_{T}  \\
    &\qquad\quad  + 
            (\Pi_{k+1} \nabla w-I_{\t{BDM}} \nabla w), \nabla e_0 -  \nabla_w e_h)_{T}) \\
    &\le C(h^{k+1} |u|_{k+3}h^2|w|_2 + h^{k+2} |u|_{k+3}h |w|_2 + h|w|_2 \|\nabla_w (Q_h u-u_h)\|_0) \\
    &\le Ch^{k+3} |u|_{k+3} \|e_h\|_0,
}where \eqref{h1e} and \eqref{reg} are applied in the end.
The proof is complete.  
                
\end{proof}

\section{Numerical test} 

We solve a 2D and a 3D Poisson's equation.
In the first example, we solve \eqref{pde} on the unit square domain $\Omega=(0,1)\times (0,1)$, 
  where 
\a{ f=2^5( y-y^2+x-x^2).  }
Thus the exact solution of the first example is
\an{\label{s-1} u=2^4(x-x^2)(y-y^2). }
We use two families of meshes shown in Figure \ref{f-g-2}.
The top meshes are quasi-uniform and the bottom meshes are degenerate which 
  violate the maximum angle condition when $h\to 0$.

\begin{figure}[H]
 \begin{center}\setlength\unitlength{1.0pt}
\begin{picture}(280,184)(0,53)
  \put(0,95){\includegraphics[width=95pt]{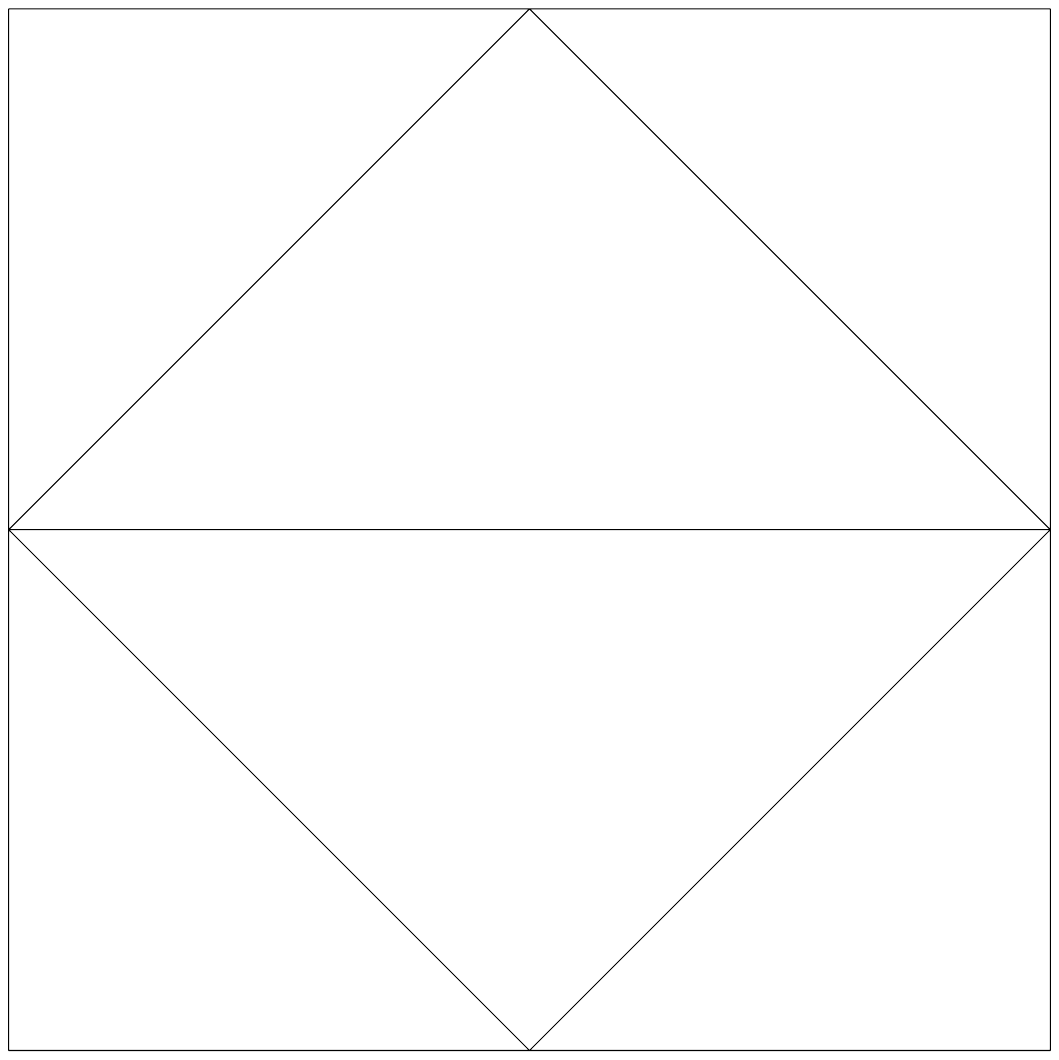}} 
  \put(95,95){\includegraphics[width=95pt]{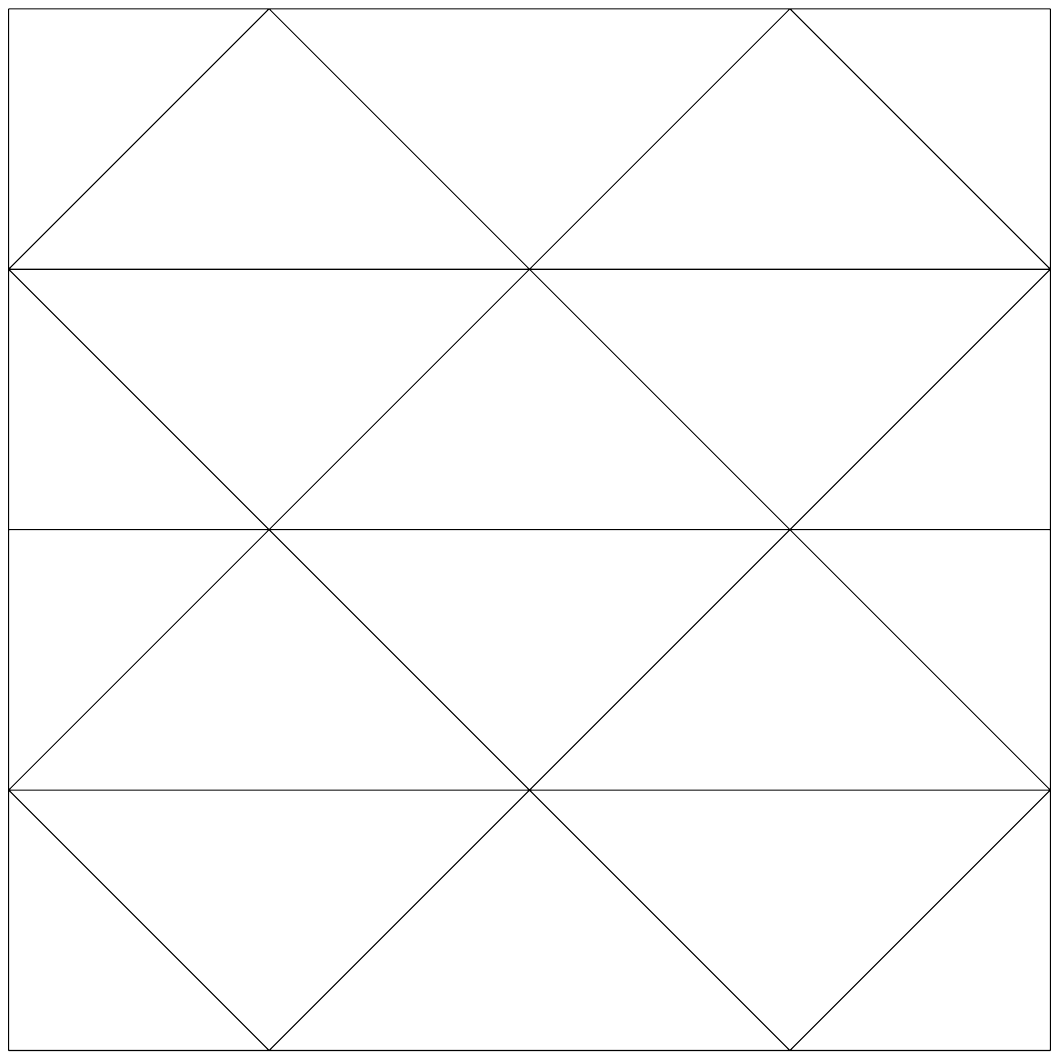}}
  \put(190,95){\includegraphics[width=95pt]{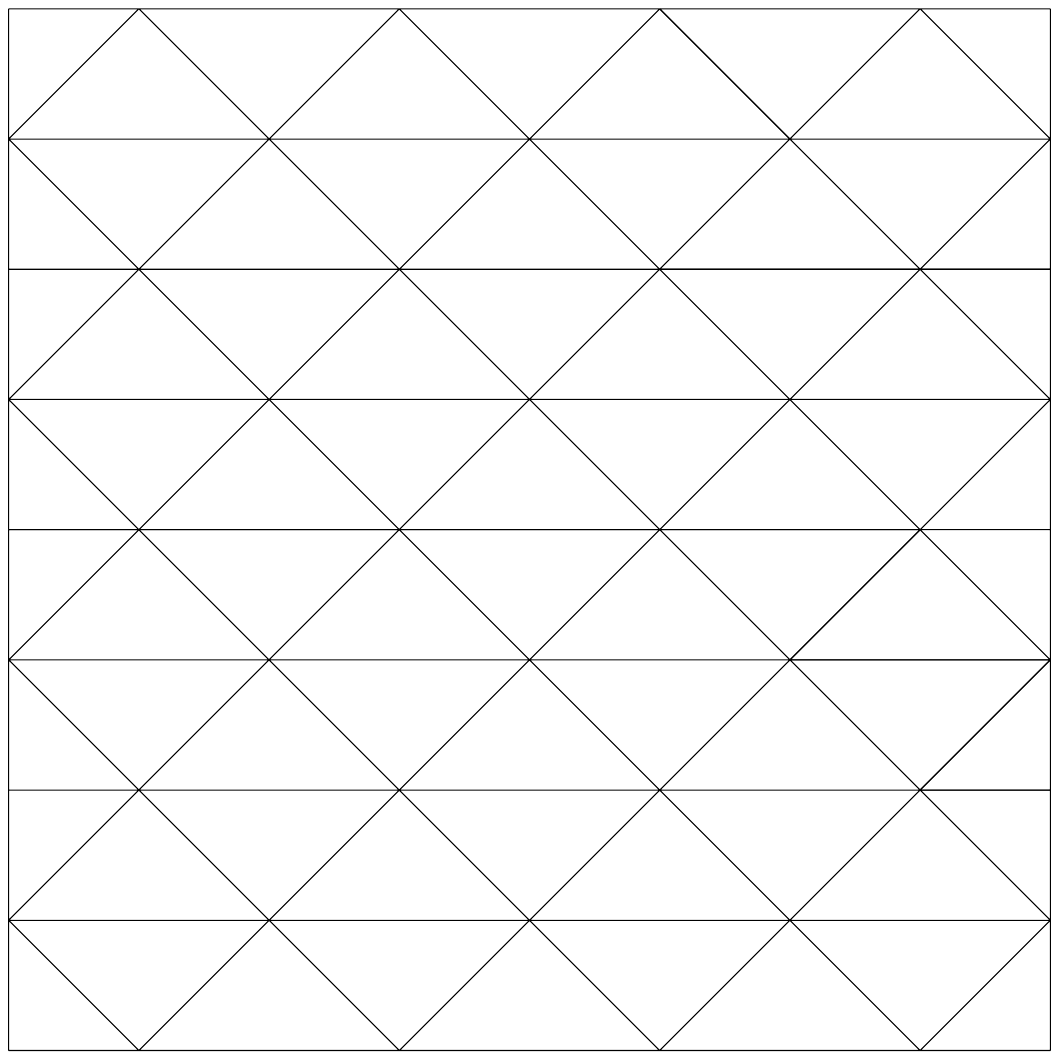}}
  \put(0, 0){\includegraphics[width=95pt]{g21grid.pdf}} 
  \put(95, 0){\includegraphics[width=95pt]{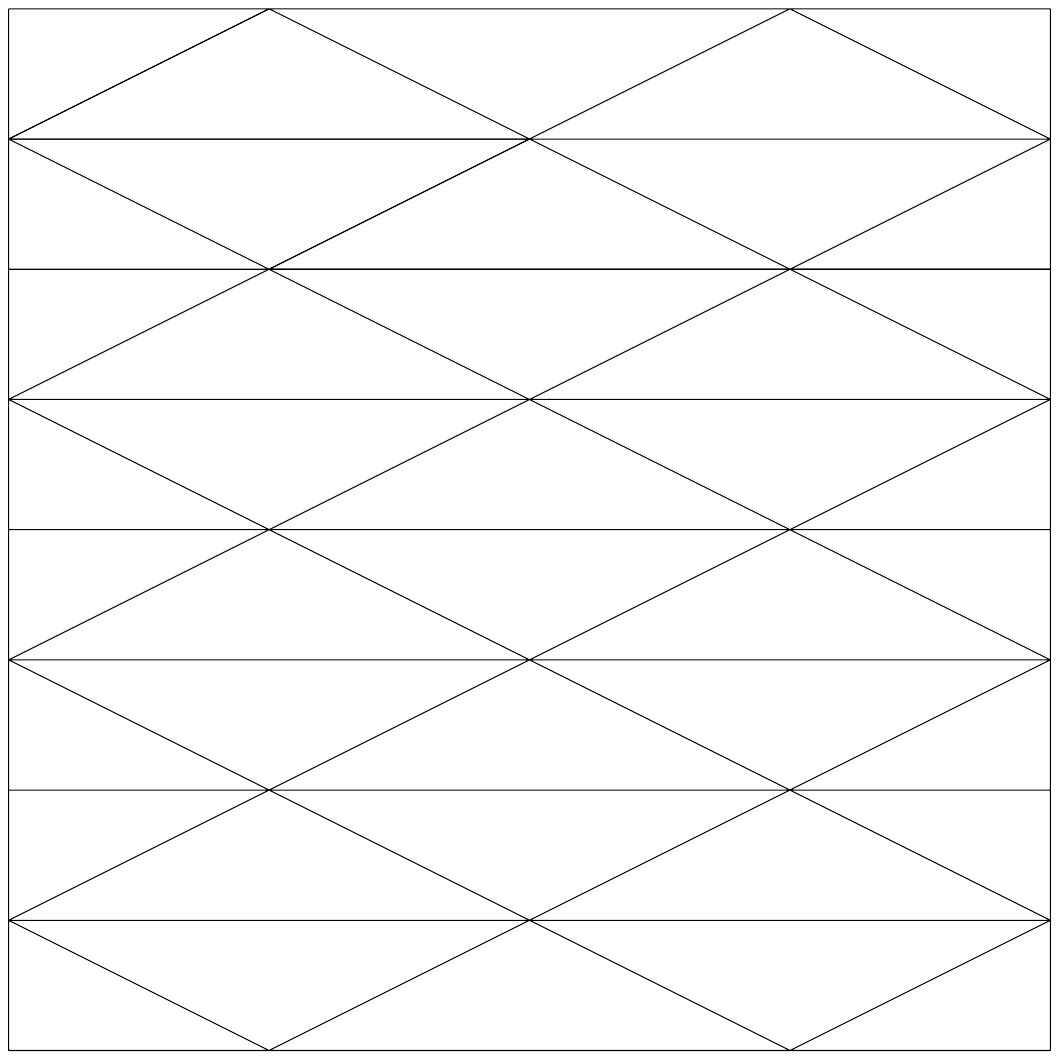}}
  \put(190, 0){\includegraphics[width=95pt]{g22grid3.pdf}}
  
  \put(5,229){\small(A)}\put(5,134){\small(B)}
 \end{picture}\end{center}
\caption{(A) Quasi-uniform but non-nested meshes.
  \quad (B) Degenerate (violates the maximum angle condition) meshes. }
\label{f-g-2}
\end{figure}

In Table \ref{t1}, we list the result for the $P_1$ conforming finite element method in
  solving the Poisson equation \eqref{pde} with the exact solution \eqref{s-1}.
 We can see the method converges at the optimal order in $L^2$ norm, and superconvergent in
   semi-$H^1$ norm, on the quasi-uniform meshes, shown in Figure \ref{f-g-2}.
However, when we refine the mesh more in one direction, it is supposedly that
   we have better solutions.
But the $P_1$ conforming finite element solutions do not converge at all, in the second half
   of Table \ref{t1}.
The solutions converge to a different function, below the true solution \eqref{s-1}.

  \begin{table}[H]
  \caption{ Error profile for the 2D $P_1$ conforming 
      finite element in computing \eqref{s-1}.} \label{t1}
\begin{center}  
   \begin{tabular}{c|rr|rr}  
 \hline 
grid &  $ \|I_h\b u - \b u_h \|_{0}$ & $O(h^r)$ &  $ |I_h\b u - \b u_h |_{1}$ & $O(h^r)$   \\ \hline 
   &\multicolumn{4}{c}{On quasi-uniform meshes Figure \ref{f-g-2}(A). }\\ \hline 
 3& 0.01306&1.6&0.07255&1.5 \\
 4& 0.00356&1.9&0.02418&1.6 \\
 5& 0.00092&1.9&0.00814&1.6 \\
 \hline 
 &\multicolumn{4}{c}{On degenerate meshes Figure \ref{f-g-2}(B).  }\\ \hline  
 3& 0.12234&0.0&0.73861&0.0\\ 
 4& 0.12556&0.0&0.80627&0.0\\ 
 5& 0.12640&0.0&0.84195&0.0\\ 
\hline 
\end{tabular} \end{center}  \end{table}

In Table \ref{t2}, we compute the $P_1$ weak Galerkin finite element ($k=1$ in \eqref{V-h}) solutions for
   the exact solution \eqref{s-2}, on the meshes shown in Figure \ref{f-g-2}.
 We can see the method converges at two orders above the optimal order in both
    $L^2$ norm and the energy norm, on both 
     the quasi-uniform meshes and the degenerate meshes in Figure \ref{f-g-2}.

  \begin{table}[H]
  \caption{ Error profile for the 2D $P_1$ weak Galerkin 
      finite element in computing \eqref{s-1}.} \label{t2}
\begin{center}  
   \begin{tabular}{c|rr|rr}  
 \hline 
grid &  $ \|I_h\b u - \b u_h \|_{0}$ & $O(h^r)$ &  $ |I_h\b u - \b u_h |_{1}$ & $O(h^r)$   \\ \hline 
   &\multicolumn{4}{c}{On quasi-uniform meshes Figure \ref{f-g-2}(A). }\\ \hline 
 3&   0.649E-04 &3.89&   0.245E-02 &2.87 \\
 4&   0.430E-05 &3.91&   0.321E-03 &2.93 \\
 5&   0.279E-06 &3.95&   0.410E-04 &2.97 \\
 \hline 
 &\multicolumn{4}{c}{On degenerate meshes Figure \ref{f-g-2}(B).  }\\ \hline  
 3&   0.549E-04 &3.80&   0.110E-02 &3.40 \\
 4&   0.376E-05 &3.87&   0.117E-03 &3.24 \\
 5&   0.247E-06 &3.93&   0.136E-04 &3.10 \\
\hline 
\end{tabular} \end{center}  \end{table}

In the second example, we solve \eqref{pde} on the unit cube domain $\Omega=(0,1)\times (0,1)\times (0,1)$, 
  where 
\a{ f=2^7((y-y^2+x-x^2)(z-z^2)+(x-x^2)(y-y^2)).  }
Thus the exact solution of the second example is
\an{\label{s-2} u=2^6(x-x^2)(y-y^2)(z-z^2). }
Again we use two families of 3D meshes shown in Figure \ref{f-g-3}.
The top meshes are quasi-uniform and the bottom meshes are degenerate.

\begin{figure}[H]
 \begin{center}\setlength\unitlength{1.0pt}
\begin{picture}(280,184)(0,53)
  \put(0,95){\includegraphics[width=95pt]{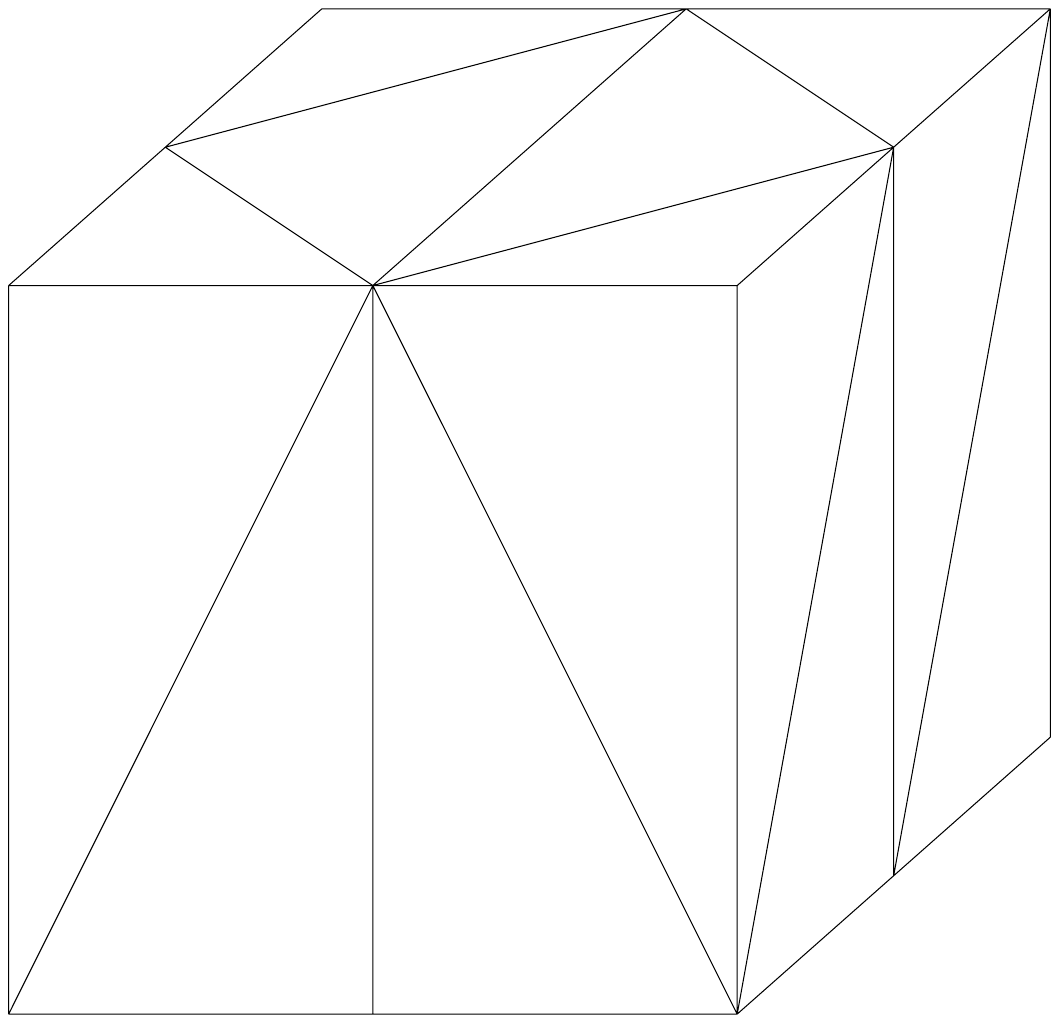}} 
  \put(95,95){\includegraphics[width=95pt]{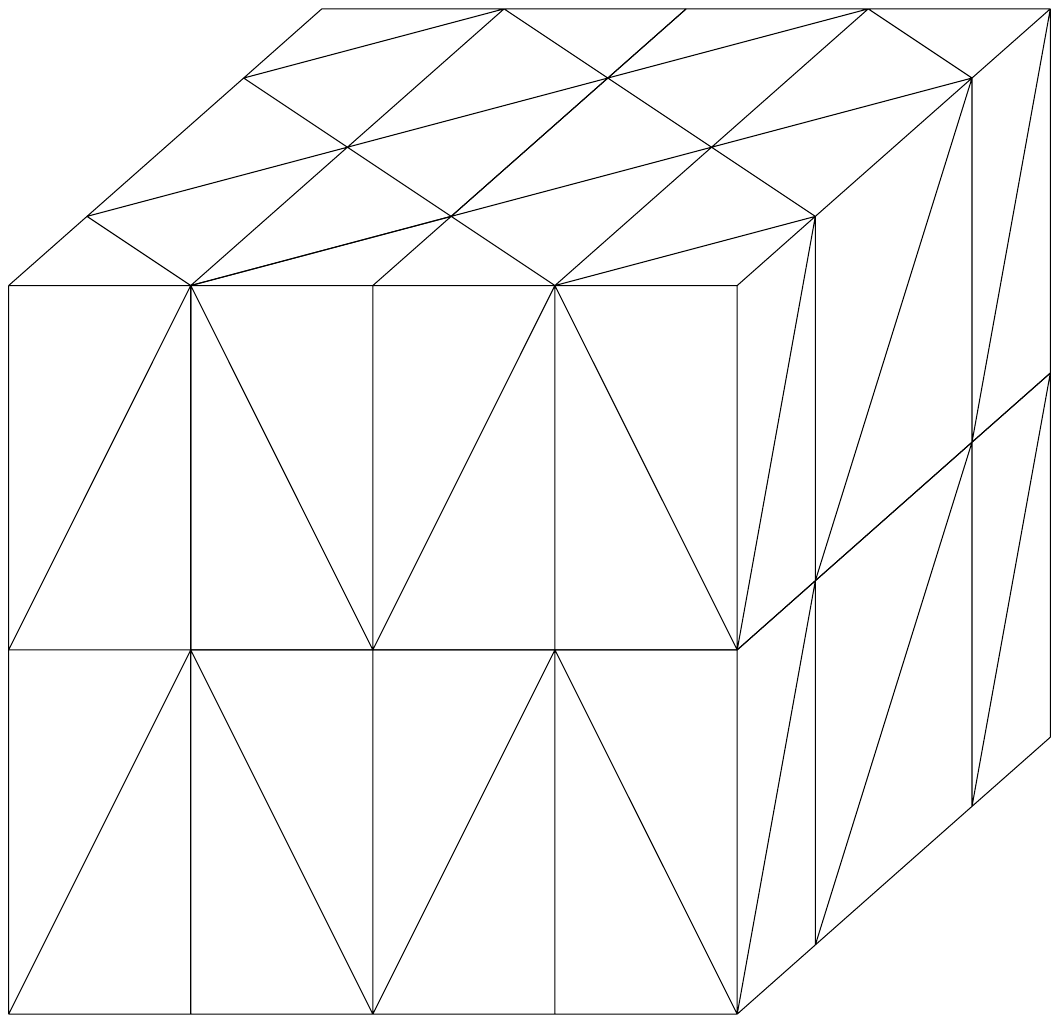}}
  \put(190,95){\includegraphics[width=95pt]{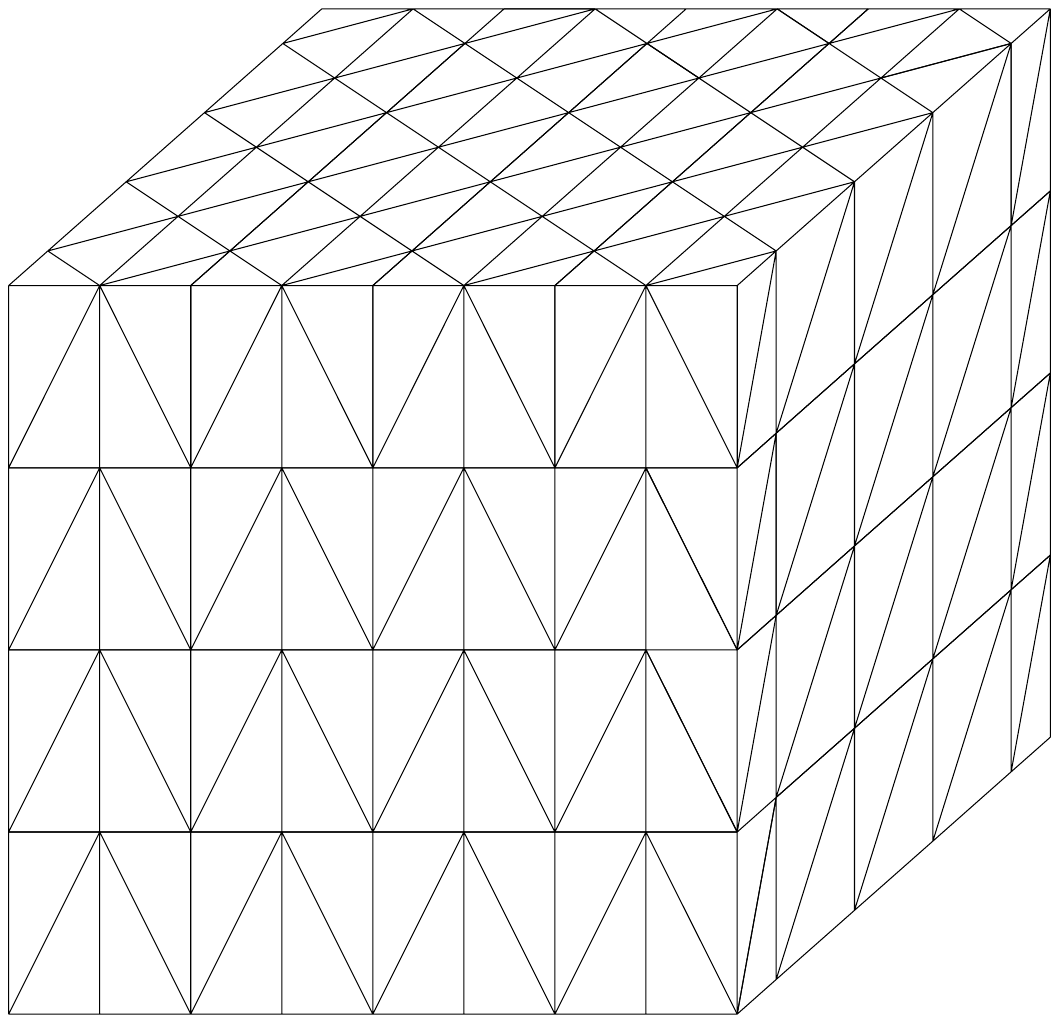}}
  \put(0, 0){\includegraphics[width=95pt]{grid311.pdf}} 
  \put(95, 0){\includegraphics[width=95pt]{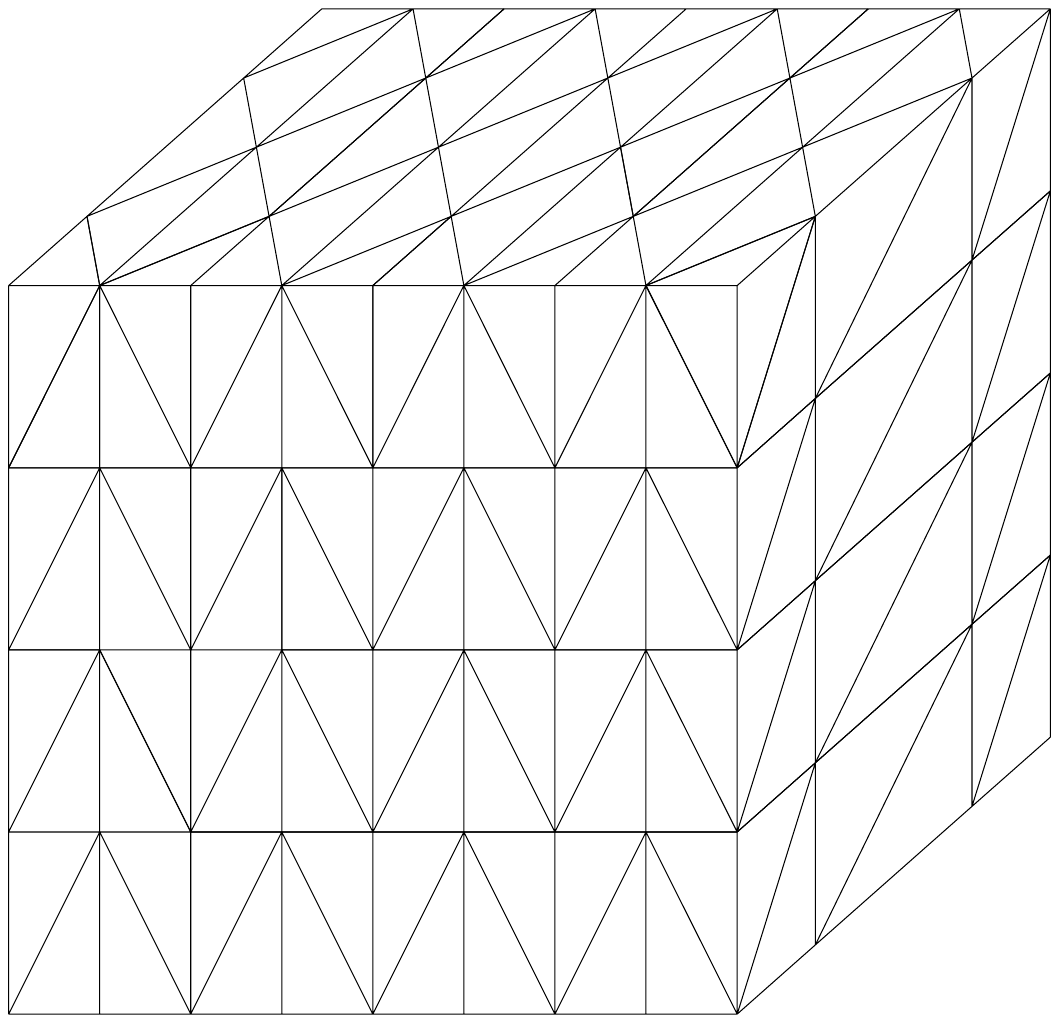}}
  \put(190, 0){\includegraphics[width=95pt]{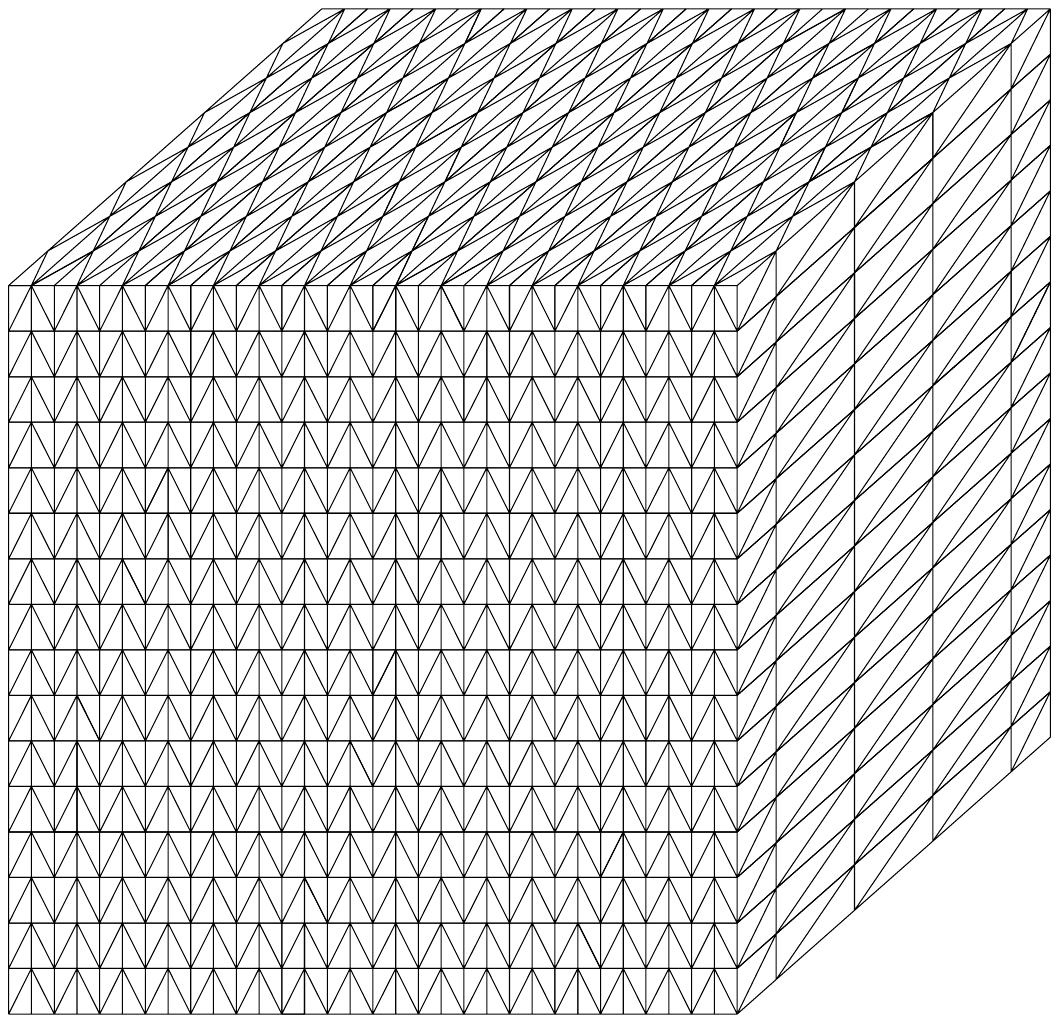}}
  
  \put(5,229){\small(A)}\put(5,134){\small(B)}
 \end{picture}\end{center}
\caption{(A) Quasi-uniform but non-nested 3D meshes.
  \quad (B) Degenerate (violates the maximum angle condition) 3D meshes. }
\label{f-g-3}
\end{figure}

In Table \ref{t1}, we list the result for the $P_1$ conforming finite element method in
  solving the Poisson equation \eqref{pde} with the exact solution \eqref{s-1}.
 We can see the method converges at the optimal order in $L^2$ norm, and superconvergent in
   semi-$H^1$ norm, on the quasi-uniform meshes, shown in Figure \ref{f-g-2}.
However, when we refine the mesh more in one direction, it is supposedly that
   we have better solutions.
But the $P_1$ conforming finite element solutions do not converge at all, in the second half
   of Table \ref{t1}.
The solutions converge to a different function, below the true solution \eqref{s-1}.

  \begin{table}[H]
  \caption{ Error profile for the 3D $P_1$ conforming 
      finite element in computing \eqref{s-2}.} \label{t3}
\begin{center}  
   \begin{tabular}{c|rr|rr}  
 \hline 
grid &  $ \|I_h\b u - \b u_h \|_{0}$ & $O(h^r)$ &  $ |I_h\b u - \b u_h |_{1}$ & $O(h^r)$   \\ \hline 
   &\multicolumn{4}{c}{On quasi-uniform meshes Figure \ref{f-g-3}(A). }\\ \hline 
 3&   0.461E-01 &1.14&   0.286E+00 &1.38 \\
 4&   0.139E-01 &1.73&   0.839E-01 &1.77 \\
 5&   0.368E-02 &1.91&   0.230E-01 &1.86 \\
 \hline 
 &\multicolumn{4}{c}{On degenerate meshes Figure \ref{f-g-3}(B).  }\\ \hline  
 2&   0.105E+00 &0.00&   0.756E+00 &0.00 \\
 3&   0.792E-01 &0.40&   0.578E+00 &0.39 \\
 4&   0.706E-01 &0.17&   0.560E+00 &0.05 \\
\hline 
\end{tabular} \end{center}  \end{table}

In Table \ref{t4}, we compute the 3D
   $P_1$ weak Galerkin finite element ($k=1$ in \eqref{V-h}) solutions for
   the exact solution \eqref{s-2}, on the meshes shown in Figure \ref{f-g-3}.
 We can see the method is two-order superconvergent in both
    $L^2$ norm and the energy norm, on both 
     the quasi-uniform meshes and the degenerate meshes in Figure \ref{f-g-3}.

  \begin{table}[H]
  \caption{ Error profile for the 3D $P_1$ weak Galerkin 
      finite element in computing \eqref{s-2}.} \label{t4}
\begin{center}  
   \begin{tabular}{c|rr|rr}  
 \hline 
grid &  $ \|I_h\b u - \b u_h \|_{0}$ & $O(h^r)$ &  $ |I_h\b u - \b u_h |_{1}$ & $O(h^r)$   \\ \hline 
   &\multicolumn{4}{c}{On quasi-uniform meshes Figure \ref{f-g-3}(A). }\\ \hline 
 3&   0.317E-03 &3.73&   0.131E-01 &2.84 \\
 4&   0.213E-04 &3.90&   0.171E-02 &2.94 \\
 5&   0.137E-05 &3.96&   0.218E-03 &2.97 \\
 \hline 
 &\multicolumn{4}{c}{On degenerate meshes Figure \ref{f-g-3}(B).  }\\ \hline  
 2&   0.109E-02 &5.54&   0.352E-01 &4.26 \\
 3&   0.475E-04 &4.53&   0.232E-02 &3.92 \\
 4&   0.293E-05 &4.02&   0.166E-03 &3.81 \\
\hline 
\end{tabular} \end{center}  \end{table}

\section{Ethical Statement}

\subsection*{Compliance with Ethical Standards} { \ }
 
   The submitted work is original and is not published elsewhere in any form or language.

This article does not contain any studies involving animals.
This article does not contain any studies involving human participants.

\subsection*{Availability of supporting data}  
   Data sharing is not applicable to this article since no datasets were generated or collected 
 in the work.

\subsection*{Competing interests} 
All authors declare that they have no potential conflict of interest.

\subsection*{Funding}

Ran Zhang was supported in part by China Natural National Science Foundation (22341302), 
the National Key R\&D Program of China \ (2023YFA 1008803 ).

\subsection*{Authors' contributions}
All authors made equal contribution.

\end{document}